\documentclass[11pt]{amsart}

\usepackage[T1]{fontenc}
\usepackage[utf8]{inputenc}
\usepackage[english]{babel}
\usepackage{amsmath,amssymb,amsthm} 
\usepackage{enumitem}
\usepackage{color}
\usepackage[margin=1in]{geometry}
\usepackage{pgfplots}
\usepackage{hyperref}
\usepackage{mathtools}
\usepgfplotslibrary{fillbetween}
\usepackage{float}
\usepackage{hyperref}

\pgfplotsset{compat=1.18}

\newtheorem{theorem}{Theorem}[section]
\newtheorem{lemma}[theorem]{Lemma}
\newtheorem{proposition}[theorem]{Proposition}

\theoremstyle{definition}

\newtheorem{remark}[theorem]{Remark}

\newtheorem*{remark*}{Remark}

\renewcommand{\Re}{\operatorname{Re}}

\newcommand{\diag}{\operatorname{diag}}

\newcommand{\superrho}{{C}_\rho^{(n)}}
\newcommand{\sbd}{c_{\rho}}
\newcommand{\mbd}{b_{\rho,K}}

\renewcommand{\MR}[1]{}

\title{Complete functional calculus bounds for $\rho$-contractions}
\author[M. Hartz]{Michael Hartz}
\author[J. de Vries]{Jens de Vries}
\address{Fachrichtung Mathematik, Universit\"at des Saarlandes, 66123 Saarbr\"ucken, Germany}
\email{hartz@math.uni-sb.de}
\email{jens\_de\_vries@math.uni-sb.de}
\date{\today}


\begin{document}

\begin{abstract}
    Let $T$ be a $\rho$-contraction on a Hilbert space. We establish sharp bounds for the functional calculus of $T$ for matrix-valued analytic functions on the unit disc. These are complete versions of bounds established by Drury and by Schwenninger and the second author.
\end{abstract}

\maketitle

\section{Introduction}

Let $\rho > 0$.
An operator $T$ on a Hilbert space $\mathcal{H}$ is said to be of class $C_\rho$ if there exists a unitary operator $U$ on a Hilbert space $\mathcal{K}$ containing $\mathcal{H}$ such that
\begin{equation*}
  T^n = \rho P_{\mathcal{H}} U^n|_{\mathcal{H}} \qquad \text{for all } n \ge 1.
\end{equation*}
Operators of class $C_\rho$ can be characterized intrinsically: $T \in C_\rho$ if and only if
\begin{equation}
  \label{eq:rho_contraction_characterization}
  I - 2 \Big( 1 - \frac{1}{\rho} \Big) \Re(\overline{z} T) - \Big(\frac{2}{\rho} - 1 \Big) |z|^2 T^* T \ge 0
  \qquad \text{ for all }z \in \overline{\mathbb{D}}.
\end{equation}
This characterization, along with other basic facts about the class $C_\rho$, can be found in \cite[Section I.11]{SFB+10}.
In particular, if $\rho = 1$, then $C_\rho$ is the class of all operators of norm at most $1$,
and if $\rho = 2$, then $C_\rho$ is the class of all operators with numerical radius at most $1$.

Let $T \in C_\rho$ and temporarily assume $\rho \ge 1$.
A theorem of Okubo and Ando \cite{OA75} shows that the closed unit disc $\overline{\mathbb{D}}$ is a $\rho$-spectral set for $T$, meaning that for every function $f \in \mathcal{O}(\overline{\mathbb{D}})$, the algebra of functions holomorphic in a neighborhood of $\overline{\mathbb{D}}$, we have
\begin{equation}
  \label{eq:rho_spectral_set}
  \|f(T)\| \le \rho \sup_{z \in \overline{\mathbb{D}}} |f(z)|.
\end{equation}
The constant $\rho$ is optimal, as the example of
\begin{equation*}
  T = \begin{bmatrix}
    0 & \rho \\ 0 & 0
  \end{bmatrix}
\end{equation*}
and $f(z) = z$ shows. Nonetheless, the bound in \eqref{eq:rho_spectral_set} can be refined if one takes into account the value of $f$ at $0$. More precisely, Schwenninger and the second author \cite{SV25} showed that for every
$f \in \mathcal{O}(\overline{\mathbb{D}})$ with $\|f\|_\infty \le 1$, we have
\begin{equation}
  \label{eq:refined_bound_scalar_case}
  \|f(T)\| \le \sbd(|f(0)|^2),
\end{equation}
where $\sbd\colon [0,1] \to [1,\rho]$ is an explicit continuous decreasing function satisfying $\sbd(0) = \rho$ and $\sbd(1) = 1$, namely
\begin{equation*}
  \sbd(t) = \frac{\rho}{2}(1-t) + \sqrt{\frac{\rho^2}{4}(1-t)^2 + t}.
\end{equation*}
In the case $\rho = 2$, this result was previously obtained by Drury \cite{Drury08}.
If $\rho = 1$, then $\sbd(t) = 1$ for all $t \in [0,1]$, so the bound in \eqref{eq:refined_bound_scalar_case} reduces to the von Neumann inequality.

Okubo and Ando \cite{OA75} in fact proved something stronger, namely that there exists an invertible operator
$S$ such that $\|S\| \|S^{-1}\| \le \rho$ and $\|S^{-1} T S\| \le 1$.
This implies that
$\overline{\mathbb{D}}$ is a complete $\rho$-spectral set for $T$, meaning that for every matrix-valued analytic function $F \in M_n(\mathcal{O}(\overline{\mathbb{D}}))$ we have
\begin{equation}
  \label{eq:complete_rho_spectral_set}
  \|F(T)\| \le \rho \sup_{z \in \overline{\mathbb{D}}} \|F(z)\|_{M_n(\mathbb{C})}.
\end{equation}
It was recently observed by the first author and M\textsuperscript{c}Carthy \cite{HM26} that the bound in \eqref{eq:complete_rho_spectral_set} can be established by a direct argument based on prior work of Clou\^atre, Ostermann and Ransford \cite{COR23}. The existence of the similarity $S$ then follows from Paulsen's similarity theorem \cite{Paulsen84}, thus giving an alternative proof of the Okubo--Ando theorem.

The goal of this note is to establish a matrix analogue of the refined bound in \eqref{eq:refined_bound_scalar_case}, i.e.\ to give upper bounds on $\|F(T)\|$ depending only on the matrix $F(0)$.
Given a matrix $A$, we will write
\begin{equation*}
  m(A)= \inf_{\|x\|=1} \|A x\| = \lambda_{\min}(A^* A)^{1/2}
\end{equation*}
for the smallest singular value of $A$.

\begin{theorem}
  \label{thm:main_intro}
  Let $\rho > 0$ and let $T \in C_\rho$. Let $F \in M_n(\mathcal{O}(\overline{\mathbb{D}}))$ with $\|F\|_\infty \le 1$. Then
  \begin{equation}\label{eq:main_intro}
    \|F(T)\| \le
    \begin{cases}
      \sbd(\|F(0)\|^2) & \text{if } \rho \le 1,\\
      \sbd(m(F(0))^2) & \text{if } \rho > 1.
    \end{cases}
  \end{equation}
\end{theorem}

Note that $\sbd$ is decreasing for $\rho \ge 1$ and increasing for $\rho \le 1$, which explains the appearance of the smallest singular value in the case $\rho \ge 1$ and the operator norm (i.e.\ the largest singular value) in the case $\rho \le 1$.

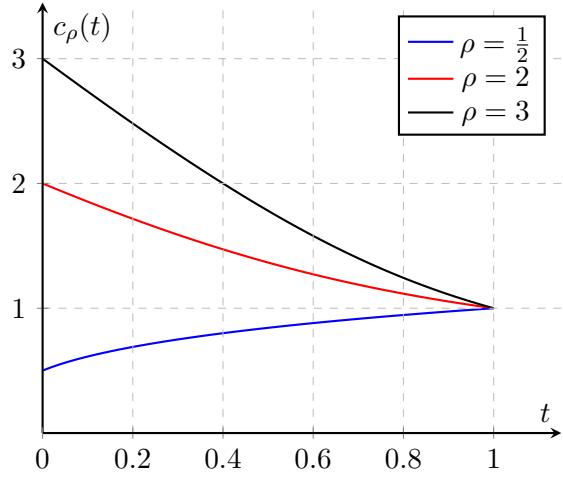
\begin{figure}[H]
\centering
\begin{tikzpicture}[declare function={
    gr(\x,\r)=\r/2*(1-\x)+sqrt(\r^2/4*(1-\x)^2+\x);
    }]
    \begin{axis}[
    axis lines = middle,
    axis on top = true,
    xlabel = {$t$},
    ylabel = {$\sbd(t)$},
    domain = 0:1,
    samples = 100,
    xmin = 0, xmax = 1.15,
    ymin = 0, ymax = 3.45,
    extra x ticks = {0},
    extra x tick labels = {$0$},
    extra x tick style = {grid = none},
    grid = both,
    grid style = {thin, dashed, gray!50},
    legend pos = north east,
    thick
    ]
    \addplot[blue, thick] {gr(x,0.5)};
    \addlegendentry{$\rho=\frac{1}{2}$}
    \addplot[red, thick] {gr(x,2)};
    \addlegendentry{$\rho=2$}
    \addplot[black, thick] {gr(x,3)};
    \addlegendentry{$\rho=3$}
    \end{axis}
\end{tikzpicture}
\caption{Graphs of $\sbd$ for $\rho=\frac{1}{2},2,3$.}
\end{figure}

We remark that in the case $\rho \ge 1$, Theorem \ref{thm:main_intro} gives (yet another) proof of the bound \eqref{eq:complete_rho_spectral_set}
of Okubo and Ando, and hence of their similarity result, thanks to Paulsen's similarity theorem.
Theorem \ref{thm:main_intro} is in fact a consequence of the following more precise result,
which provides an operator upper bound
on $F(T)^* F(T)$ that only depends on $F(0)$.

\begin{theorem}\label{thm:main_result}
    Let $\rho>0$ and let $T\in C_{\rho}$. Let $F\in M_n(\mathcal{O}(\overline{\mathbb{D}}))$ with $\|F\|_\infty\leq1$.
    Let $A=F(0)$.
    \begin{enumerate}[label=(\alph*)]
        \item If $0<\rho\leq1$, let $K=\sbd(\|A\|^2)$ and define the increasing affine function
          $\mbd\colon [0,1] \to [\rho^2,1]$ by
        \begin{equation*}
            \mbd(t)=\rho K-(\rho K-1)t.
        \end{equation*}
        Then
        \begin{equation*}
            F(T)^*F(T)\leq \mbd(A^*A).
        \end{equation*}
        \item If $\rho>1$, let $K=\sbd(m(A)^2)$
          and define the decreasing continuous function
          $\mbd\colon [0,1] \to [1, \rho^2]$ by
        \begin{equation*}
            \mbd(t)=\begin{cases}
                \rho K - (\rho K-1)t& \text{if } 0\leq t\leq \big(\frac{(\rho-1)K}{\rho K-1}\big)^2,\\
                1+\dfrac{(K-1)^2(1+\sqrt{t})}{2(K-1)-(\rho K-1)(1-\sqrt{t})}& \text{if } \big(\frac{(\rho-1)K}{\rho K-1}\big)^2<t \le 1.
            \end{cases}.
        \end{equation*}
        Then
        \begin{equation*}
            F(T)^*F(T)\leq \mbd(A^*A).
        \end{equation*}
    \end{enumerate}
\end{theorem}

At the most basic level, our strategy to prove Theorem \ref{thm:main_result} is similar to that in \cite{SV25}:
The case $F(0) = 0$ is easy to deal with using the unitary $\rho$-dilation of $T$ and to understand the general case it remains to study the action of suitable M\"obius transforms on operators of class $C_\rho$. However, the arguments in the complete setting differ substantially from those in the scalar case. Since we are dealing with matrix-valued functions, we need to work with
Potapov--Möbius transforms on the unit ball of $M_n(\mathbb{C})$; cf.\ \cite{HM26}.
It turns out that the appropriate class of operators on $\mathcal{H}^n$ is not $C_\rho$ itself, but a subclass of $C_\rho$ that is stable under multiplication by scalar matrices, which we call $\superrho$.
The main burden of the proof is then to analyze the action of Potapov--Möbius transforms on operators in $\superrho$.
Here, new phenomena arise that are not present in the scalar case, as can be seen from the statement of Theorem \ref{thm:main_result} and the piecewise definition of $\mbd$.

Finally, we address the sharpness of our bounds.
Sharpness of Theorem \ref{thm:main_intro} follows essentially from sharpness of the scalar bound
\eqref{eq:refined_bound_scalar_case} established in \cite{SV25}. More precisely,
we show that for every matrix $A\in M_{n}(\mathbb{C})$ with $\|A\| \le 1$, we can find an operator $T\in C_{\rho}$ and a function $F\in M_{n}(\mathcal{O}(\overline{\mathbb{D}}))$ with $\|F\|_{\infty}=1$ and $F(0)=A$ such that equality is attained in \eqref{eq:main_intro};
see Proposition \ref{prop:sharpness_main_intro}.
The question of sharpness of Theorem \ref{thm:main_result} is more subtle as the order on Hermitian matrices is not a lattice order; we will discuss this in Section \ref{sec:sharpness}. Nonetheless, we will show that the bound in Theorem \ref{thm:main_result} is optimal under some natural assumptions on the shape of the upper bound;
see Theorem \ref{thm:sharpness_main_result} for the precise statement.

\subsection*{Acknowledgement}
The authors are grateful to John M\textsuperscript{c}Carthy for helpful discussions.

\subsection*{AI disclosure}
The proofs of Lemma \ref{lem:PSD_reduction_full} and Proposition \ref{prop:main_result_K} were obtained with the help of ChatGPT, extending human generated arguments in special cases. Part of the construction in the proofs of Theorem \ref{thm:sharpness_main_result} and Lemma \ref{lem:weighted_shift_C_rho} are based on arguments provided by ChatGPT. 
We also used ChatGPT for preliminary computations, typesetting, generating figures and proofreading,
as well as GitHub Copilot for autocomplete.

\section{The class $\superrho$}

Let $F \in M_n(\mathcal{O}(\overline{\mathbb{D}}))$ be a matrix-valued holomorphic function
with $\|F\|_\infty \le 1$ and assume that $A = F(0)$ satisfies $\|A\| < 1$.
Using Potapov--M\"obius transforms $M_A$, which we will explain in detail in the next section, we can write $F = M_A \circ G$, where $G\colon \overline{\mathbb{D}} \to M_n(\mathbb{C})$ is a holomorphic function with $\|G\|_\infty \le 1$ and $G(0) = 0$.
As in the scalar case \cite{SV25}, we will analyze the actions of $G$ and of $M_A$ on operators of class $C_\rho$ separately.
In this section, we will deal with the former.

As usual, we identify $B(\mathcal{H}^n) = M_n(B(\mathcal{H})) = M_n(\mathbb{C}) \otimes B(\mathcal{H})$.
Moreover, if $X \in M_n(\mathbb{C})$ is a scalar matrix, we can think of $X$ as an operator on $\mathcal{H}^n$ by identifying $X$ with $X \otimes I_{\mathcal{H}}$.
We start with the following lemma.

\begin{lemma}
  \label{lem:vanishing_at_zero}
  Let $\rho > 0$.
  Let $S \in C_\rho$ and let $G \in M_n(\mathcal{O}(\overline{\mathbb{D}}))$ satisfy
  $\|G\|_\infty \le 1$ and $G(0) = 0$.
  Let $T = G(S) \in B(\mathcal{H}^n)$. Then for all scalar matrices $X,Y \in M_n(\mathbb{C})$ with $\|X\| \le 1$
  and $\|Y\| \le 1$, the operator $X T Y \in B(\mathcal{H}^n)$ is of class $C_\rho$.
\end{lemma}

\begin{proof}
  The proof is a straightforward adaptation of the proof in the scalar case \cite[Proposition I.11.5]{SFB+10}. 
  Indeed, let $U$ be a unitary $\rho$-dilation of $S$ on a Hilbert space $\mathcal{K}$ containing $\mathcal{H}$. Then, because $G(0)=0$, we have
  \begin{align*}
      G(S)^{k}=\rho P_{\mathcal{H}^{n}}G(U)^{k}|_{\mathcal{H}^{n}}
  \end{align*}
  for all integers $k\geq1$. Since $U$ is a unitary and $\|G\|_{\infty}\leq1$, it follows that $G(U)$ is contractive. Hence, by Sz.-Nagy's dilation theorem, there exists a unitary $1$-dilation $V$ of $G(U)$ on a Hilbert space $\mathcal{L}$ containing $\mathcal{K}^{n}$. This means that
  \begin{align*}
      G(U)^{k} = P_{\mathcal{K}^{n}}V^{k}|_{\mathcal{K}^{n}}.
  \end{align*}
  Combining the dilations shows that $V$ is unitary $\rho$-dilation of $G(S)$ on $\mathcal{L}$, so $G(S)$ is of class $C_\rho$. Since this argument also applies to the function $z \mapsto X G(z) Y$, the general statement follows.
\end{proof}

We let
\begin{equation*}
  \superrho = \{ T \in B(\mathcal{H}^n): X T Y \in C_\rho \text{ for all } X,Y \in M_n(\mathbb{C}) \text{ with } \|X\|, \|Y\| \le 1 \}.
\end{equation*}
As the next remark shows, $\superrho$ is a proper subset of $C_\rho$ in general. The extra strength
of the conclusion $G(S) \in \superrho$ in Lemma \ref{lem:vanishing_at_zero} will turn out to be vital.

\begin{remark}
  Let $\rho \ge 1$.
    If $T=M\otimes I_{\mathcal{H}}$ for some $M\in M_{n}(\mathbb{C})$, then $T\in \superrho$ if and only if $\|T\|\leq1$. Indeed, if $T\in \superrho$ and $U\in M_{n}(\mathbb{C})$ is unitary such that $M=U|M|$, then $U^* T=|M|\otimes I_{\mathcal{H}}$ is positive (hence self-adjoint) and satisfies $U^* T\in C_{\rho}$. In particular, $\|T\|=\|U^* T\|=r(U^*T)\leq1$. Conversely, if $\|T\|\leq1$, then $XTY\in C_{1}\subset C_{\rho}$ for all $X,Y\in M_{n}(\mathbb{C})$ with $\|X\|,\|Y\|\leq1$.
\end{remark}

The following result is a matrix analogue of \cite[Theorem 2.2]{SV25}. At the same time, we extend the result to $\rho < 1$.

\begin{lemma}\label{lem:class_superrho}
  Let $\rho > 0$ and let $T \in \superrho$. Then for all $X \in M_n(\mathbb{C})$ with $\|X\| \le 1$,
  we have
  \begin{equation*}
    I_{\mathcal{H}^n} - 2 \Big( 1 - \frac{1}{\rho} \Big) \Re(X^* T)
    - T^* \Big( \frac{1}{\rho^2}I - \Big(1 - \frac{1}{\rho} \Big)^2 X X^* \Big) T \ge 0.
  \end{equation*}
\end{lemma}

\begin{proof}
  The proof is similar to the proof of \cite[Theorem 2.2]{SV25}. Since the matrix-valued case and the case $\rho < 1$ require some extra care, we provide a few more details.

  Let
  \begin{equation*}
    L_X(T) = I_{\mathcal{H}^n} - 2 \Big( 1 - \frac{1}{\rho} \Big) \Re(X^* T)
    - T^* \Big( \frac{1}{\rho^2}I - \Big(1 - \frac{1}{\rho} \Big)^2 X X^* \Big) T.
  \end{equation*}
  Given $X \in M_n(\mathbb C)$ with $\|X\| \le 1$, let $Y = \frac{1}{\rho} I + (1 - \frac{1}{\rho}) X^*$. We claim that the spectral radius of $Y T$ is most $1$.
  If $\rho \ge 1$, then $\|Y\| \le 1$, so $Y T \in C_\rho$, which implies
  the spectral radius bound. If $\rho < 1$, then $\|Y\| \le \frac{2 - \rho}{\rho}$,
  so $\frac{\rho}{2 - \rho} Y T \in C_\rho$. But the spectral
  radius of an operator of class $C_\rho$ is at most $\frac{\rho}{2 - \rho}$;
  see \cite[Lemma 5]{BS67} or \cite[Theorem 1]{Furuta68a}.
  Hence the spectral radius of $Y T$ is at most $1$ in all cases.
  Moreover, by replacing $T$ with $r T$ for $0 < r <1$, we may even assume
  that this inequality strict.
  
  Hence, we may define
  \begin{equation*}
    R(X) = \Big( I_{\mathcal{H}^n} - \Big( \frac{1}{\rho}I + \Big( 1 - \frac{1}{\rho} \Big) X^* \Big) T \Big)^{-1}
  \end{equation*}
  for $X\in M_n(\mathbb{C})$ with $\|X\|\leq1$.
  A computation shows that
  \begin{equation*}
    \Re\Big( I_{\mathcal{H}^n} + \frac{2}{\rho} T R(X)\Big) = R(X)^* L_X(T) R(X).
  \end{equation*}
  Hence $L_X(T) \ge 0$ if and only if $\Re( I_{\mathcal{H}^n} + \frac{2}{\rho} T R(X)) \ge 0$.
  The map $X \mapsto \Re( I_{\mathcal{H}^n} + \frac{2}{\rho} T R(X))$ is pluriharmonic,
  so using a singular value decomposition of $X$ and the minimum principle for harmonic functions $n$ times,
  we see that $\Re(I_{\mathcal{H}^n} + \frac{2}{\rho}T R(X)) \ge 0$ for all $X \in M_n(\mathbb{C})$
  with $\|X\| \le 1$ if and only if this holds for all unitary $X$.
  Going back to $L_X(T)$, we see that for unitary $X$, the operator $L_X(T)$ is positive by \eqref{eq:rho_contraction_characterization} since $X^* T \in C_\rho$, as $T \in \superrho$.
\end{proof}

\section{Potapov--Möbius transforms}

We start by recalling a few facts about Potapov--M\"obius transforms.
If $A,X \in B(\mathcal{H})$ such that $\|A\| < 1$ and $I + A^* X$ is invertible, let
\begin{equation*}
  M_A(X) = (I - A A^*)^{-1/2} (X + A)(I + A^* X)^{-1} (I - A^* A)^{1/2}.
\end{equation*}
The map $M_{A}$ sends the closed unit ball of $B(\mathcal H)$ into itself, see e.g.\ \cite[Theorem 8.19]{ISL85}.
As in \cite{HM26}, we will use Potapov--M\"obius transforms in the case when the Hilbert space decomposes
as $\mathcal{H}^n$, so that we have identifications $B(\mathcal{H}^n) = M_n(B(\mathcal{H})) = M_n(\mathbb{C}) \otimes B(\mathcal{H})$.
Given a scalar matrix $A \in M_n(\mathbb{C})$ with $\|A\| < 1$, we identify $A$ with $A \otimes I_{\mathcal{H}} \in B(\mathcal{H}^n)$
and thus define $M_A(T)$ for $T \in B(\mathcal{H}^n)$ for which $I + A^* T =  I_{\mathcal{H}^n} + (A^* \otimes I_{\mathcal{H}}) T$
is invertible.

From Lemma \ref{lem:vanishing_at_zero}, we are led to analyzing $M_A(T)$ for operators $T \in \superrho$. To this end, we begin with the following result, which, in fact, holds for a broader class of operators.

\begin{lemma}
  \label{lem:M_A_upper_bound}
  Let $T \in B(\mathcal{H}^n)$ and $A\in M_{n}(\mathbb{C})$ with $\|A\|<1$ and assume that $I+A^* T$ is invertible. Let $B \in M_n(\mathbb{C})$ be a positive matrix that commutes with $A^{*}A$ and for which
  $B - A^* A$ is positive and invertible. Then
  \begin{equation*}
    M_A(T)^* M_A(T) \le B
  \end{equation*}
  if and only if
  \begin{equation*}
    T^* C T - 2 \Re(DT)\le E,
  \end{equation*}
  where
  \begin{align*}
    C &= (I - A A^*)^{-1/2} (I - A B A^*) (I - A A^*)^{-1/2},\\
    D &= (B-I) (I-A^* A)^{-1} A^*, \\
    E &= (I - A^* A)^{-1/2} (B - A^* A) (I - A^* A)^{-1/2}.
  \end{align*}
\end{lemma}

\begin{proof}
Note that $M_A(T)^* M_A(T)\leq B$ if and only if
\begin{equation*}
(I - A^* A)^{1/2} (I + T^* A)^{-1}(T^* + A^*)(I - A A^*)^{-1} (T + A)(I + A^* T)^{-1} (I - A^* A)^{1/2}\leq B.
\end{equation*}
Conjugating by $(I-A^* A)^{-1/2}(I+A^* T)$ yields the equivalent condition
\begin{equation*}
    (T^*+A^*)(I-AA^*)^{-1}(T+A)\leq(I+T^* A)(I-A^*A)^{-1/2}B(I-A^* A)^{-1/2}(I+A^*T).
\end{equation*}
Expanding and grouping terms by their order in $T$, this is seen to be equivalent to
\begin{align*}
    &T^*(I-AA^*)^{-1}T - T^* A(I-A^* A)^{-1/2}B(I-A^* A)^{-1/2} A^* T \\
    &\qquad {}+T^*(I-AA^*)^{-1}A + A^*(I-AA^*)^{-1}T \\
    &\qquad {}- T^*A(I-A^*A)^{-1/2}B(I-A^*A)^{-1/2} -(I-A^*A)^{-1/2}B(I-A^*A)^{-1/2}A^*T \\
    {}\leq{}  & (I-A^*A)^{-1/2}B(I-A^*A)^{-1/2}-A^*(I-AA^*)^{-1}A.
\end{align*}
Finally, since $B$ and $A^* A$ commute and $A (I - A^* A)^{-1/2} = (I - A A^*)^{-1/2} A$,
this is equivalent to
\begin{align*}
    &T^*(I-AA^*)^{-1/2}(I-ABA^*)(I-AA^*)^{-1/2}T\\
    &\qquad {}-2\Re\big((I-A^*A)^{-1/2}(B-I)(I-A^*A)^{-1/2}A^* T\big)\\
    {}\leq{}& (I-A^* A)^{-1/2}(B-A^* A)(I-A^* A)^{-1/2}. \qedhere
\end{align*}
\end{proof}

Note that the coefficients $C$, $D$ and $E$, while somewhat complicated, are independent of $T$.

\begin{remark}
    The condition that $I + A^*T$ is invertible is automatically satisfied for $T\in\superrho$ and $A\in M_{n}(\mathbb{C})$ with $\|A\|<1$. This follows from the fact that the spectrum of an operator in $C_{\rho}$ lies in the closed unit disc $\overline{\mathbb{D}}$.
\end{remark}

\begin{lemma}\label{lem:PSD_reduction_full}
    Let $\rho>0$ and $T \in \superrho$. Let $A\in M_{n}(\mathbb{C})$ with $\|A\|<1$
    and let $b\colon \sigma(A^*A) \to [0,\infty)$ be a (continuous) function
    satisfying $b(t) > t$ for all $t \in \sigma(A^*A)$.
    Define $B = b(A^*A)$. If there exist $\delta \ge 0$
    and a function $x\colon \sigma(A^*A) \to [-1,1]$ such that
    \begin{equation*}
         W(t) = \begin{bmatrix}
            \dfrac{b(t)-t}{1-t} & \dfrac{(b(t)-1)\sqrt{t}}{1-t}\\[2mm]
            \dfrac{(b(t)-1)\sqrt{t}}{1-t} & \dfrac{b(t)t-1}{1-t}
        \end{bmatrix}
        +\delta\begin{bmatrix}
            -1 & \Big(1-\dfrac{1}{\rho}\Big)x(t)\\[2mm]
            \Big(1-\dfrac{1}{\rho}\Big)x(t) & \dfrac{1}{\rho^2}-\Big(1-\dfrac{1}{\rho}\Big)^2 x(t)^2
        \end{bmatrix}
    \end{equation*}
    is positive for every $t \in \sigma(A^*A)$, then
    \begin{equation*}
        M_A(T)^*M_A(T)\leq B.
    \end{equation*}
\end{lemma}
\begin{proof}
Let $C$, $D$ and $E$ be defined as in Lemma \ref{lem:M_A_upper_bound}. By the assumption on $b$,
the matrix
\begin{equation*}
    B-A^*A=b(A^*A)-A^*A
\end{equation*}
is positive and invertible, and $B$ commutes with $A^* A$. Define
\begin{equation*}
    Q(T)=E+2\Re(DT)-T^* C T.
\end{equation*}
By Lemma \ref{lem:M_A_upper_bound}, we have that $M_A(T)^* M_A(T)\leq B$ if and only if $Q(T)\geq0$.

Choose a unitary polar decomposition $A=U(A^*A)^{1/2}$, and set
\begin{equation*}
    X=Ux(A^*A).
\end{equation*}
Since $\|X\|\leq1$ and $T\in\superrho$, Lemma \ref{lem:class_superrho} implies that
\begin{equation*}
    L_X(T)=I_{\mathcal{H}^n} - 2 \Big( 1 - \frac{1}{\rho} \Big) \Re(X^* T)
    - T^* \Big( \frac{1}{\rho^2}I - \Big(1 - \frac{1}{\rho} \Big)^2 X X^* \Big) T\geq0.
\end{equation*}
Thus it suffices to prove that $Q(T)-\delta L_X(T)\geq0$.

Let $S=U^*T$. Since $B=b(A^*A)$ and $A A^* = U A^* A U^*$, we compute
\begin{align*}
    C&=U(I-A^*A)^{-1}(I-BA^*A)U^*,\\
    D&=(B-I)(I-A^*A)^{-1}(A^*A)^{1/2}U^*,\\
    E&=(I-A^*A)^{-1}(B-A^*A).
\end{align*}
Consequently,
\begin{align*}
    Q(T)
    &=(I-A^*A)^{-1}(B-A^*A)\\
    &\qquad{}+2\Re\big((I-A^*A)^{-1}(B-I)(A^*A)^{1/2}S\big)\\
    &\qquad{}+S^* (I-A^*A)^{-1}(B A^*A-I)S,
\end{align*}
and, likewise,
\begin{equation*}
    L_X(T)=I_{\mathcal{H}^n}-2\Big(1-\frac{1}{\rho}\Big)\Re(x(A^*A)S)-S^* \Big(\frac{1}{\rho^2}I-\Big(1 - \frac{1}{\rho}\Big)^2x(A^*A)^2\Big)S.
\end{equation*}
It follows that
\begin{equation*}
    Q(T)-\delta L_X(T)
    =\begin{bmatrix}
        I_{\mathcal{H}^n} & S^*
    \end{bmatrix}W(A^*A)\begin{bmatrix}
        I_{\mathcal{H}^n} \\ S
    \end{bmatrix}.
\end{equation*}
Here, $W(A^*A)$ means the continuous functional calculus applied entrywise.
Since $W(t) \ge 0$ for all $t \in \sigma(A^*A)$, it follows that
$W(A^*A)\geq0$. Thus $Q(T)-\delta L_X(T)\geq0$, as desired.
\end{proof}

Recall that
\begin{align*}
\sbd(t) =\frac{\rho}{2}(1-t)+\sqrt{\frac{\rho^{2}}{4}(1-t)^{2}+t}.
\end{align*}
This function satisfies the quadratic equation
  \begin{equation}
    \label{eqn:quadratic_equation_c_rho}
    \sbd(t)^2 - \rho(1-t) \sbd(t) - t = 0.
  \end{equation}
The following simple computation will be used repeatedly, so we isolate it as a lemma.

\begin{lemma}
  \label{lem:threshold}
  Let $\rho > 0$ with $\rho \neq 1$ and let $t \in [0,1]$. Let $K = \sbd(t)$.
  Then
  \begin{equation*}
    t = \frac{K(\rho-K)}{\rho K-1}
    \le \frac{(\rho-1)^2 K^2}{(\rho K-1)^2}.
  \end{equation*}
\end{lemma}

\begin{proof}
  The equality follows from the quadratic equation \eqref{eqn:quadratic_equation_c_rho}.
  For the inequality, observe that
  \begin{equation*}
    \frac{(\rho-1)^2 K^2}{(\rho K-1)^2} -t
    = \frac{\rho K (K-1)^2}{(\rho K - 1)^2} \ge 0. \qedhere
  \end{equation*}
\end{proof}

We are now able to prove the desired bound on Potapov--Möbius transforms, applied to operators of class $\superrho$.
\begin{proposition}\label{prop:main_result_K}
    Let $\rho>0$ and let $T\in\superrho$. Let $A\in M_n(\mathbb{C})$ with $\|A\|<1$.
    \begin{enumerate}[label=(\alph*)]
        \item If $0<\rho\leq1$, let $K=\sbd(\|A\|^2)$ and define the increasing affine function
          $\mbd\colon [0,1] \to [\rho^2,1]$ by
        \begin{equation*}
            \mbd(t)=\rho K-(\rho K-1)t.
        \end{equation*}
        Then
        \begin{equation*}
            M_A(T)^*M_A(T)\leq \mbd(A^*A).
        \end{equation*}

        \item If $\rho>1$, let $K=\sbd(m(A)^2)$
          and define the decreasing continuous function
          $\mbd\colon [0,1] \to [1, \rho^2]$ by
        \begin{equation*}
            \mbd(t)=\begin{cases}
                \rho K-(\rho K-1)t& \text{if } 0\leq t\leq \big(\frac{(\rho-1)K}{\rho K-1}\big)^2,\\
                1+\dfrac{(K-1)^2(1+\sqrt{t})}{2(K-1)-(\rho K-1)(1-\sqrt{t})}& \text{if } \big(\frac{(\rho-1)K}{\rho K-1}\big)^2<t \le 1.
            \end{cases}.
        \end{equation*}
        Then
        \begin{equation*}
            M_A(T)^*M_A(T)\leq \mbd(A^*A).
        \end{equation*}
    \end{enumerate}
\end{proposition}
\begin{proof}
    We use Lemma \ref{lem:PSD_reduction_full}. In both cases, set $b=\mbd$ and $\delta=\rho K$. Observe that $b(t) > t$ for all $t \in \sigma(A^* A)$ because $\|A\| < 1$.
    We have to show that there exists a function $x\colon\sigma(A^*A)\to[-1,1]$ with $W(t)\geq0$ for all
    $t\in\sigma(A^*A)$, where $W(t)$ is the matrix from Lemma \ref{lem:PSD_reduction_full}.

    If $\rho=1$, then $K=1$ and $b(t)=1$ and we may simply take $x(t)=0$ for all
    $t\in\sigma(A^*A)$. Then $W(t)=0$ for all $t\in\sigma(A^*A)$, which establishes the result in the case $\rho=1$.

    Let $\rho \neq 1$. We first deal with the affine branch of $b$.
    If $t \in \sigma(A^* A)$ is such that $b(t)=\rho K-(\rho K-1)t$, then we take
    \begin{equation*}
        x(t)=-\frac{\rho K-1}{(\rho-1)K}\sqrt{t}.
    \end{equation*}
    A direct computation gives that
    \begin{equation}
      \label{eqn:W_affine_branch}
        W(t)=\begin{bmatrix}
            0&0\\
            0&\big(1-\frac{1}{\rho K}\big)(t-t_0)
        \end{bmatrix},
        \qquad \text{ where } \qquad
        t_0=\frac{K(\rho-K)}{\rho K-1}.
    \end{equation}

    We now distinguish between the cases $\rho<1$ and $\rho>1$.
    If $\rho<1$, then $1 - \frac{1}{\rho K} \leq 0$. Moreover, since $K = \sbd(\|A\|^2)$, we have
    $t_0 = \|A\|^2$ by Lemma \ref{lem:threshold}.
    Hence every $t \in \sigma(A^*A)$ satisfies
    $t\leq t_0$. Thus \eqref{eqn:W_affine_branch} implies that
    $W(t) \ge 0$ for all $t\in\sigma(A^*A)$.
    It remains to check that $x$ takes values in
    $[-1,1]$. Since $x(t)\leq0$, it suffices to show that $x(\|A\|^2)\geq-1$, which follows from Lemma \ref{lem:threshold}.
    This proves part (a).

    Suppose now that $\rho>1$.
    Then $t_0 = m(A)^2$.
    Set
    \begin{equation*}
        s_\dagger=\frac{(\rho-1)K}{\rho K-1},
    \end{equation*}
    so that $s_\dagger^2$ is the threshold between the two regimes of $\mbd$.
    Then
    \begin{equation*}
        x(t)=-\frac{\rho K-1}{(\rho-1)K}\sqrt{t}=-\frac{\sqrt{t}}{s_\dagger},
    \end{equation*}
    which shows that $x(t)\in[-1,1]$ on the affine branch. Moreover, for $t\in\sigma(A^*A)$ we have
    $t\geq m(A)^2=t_0$, and $1-\frac{1}{\rho K}\geq0$. Thus, \eqref{eqn:W_affine_branch} shows
    $W(t)\geq0$ on this branch.

    It remains to deal with the second branch of $\mbd$.
    If $t \in \sigma(A^* A)$ satisfies $t> s_\dagger^2$, then we take $x(t)=-1$.
    Write $s=\sqrt{t}$.
    A computation gives
    \begin{equation*}
        W(t)=\frac{1}{(\rho K-1)(1-s)\big((1-s_\dagger)+(s-s_\dagger)\big)}w(s)w(s)^*,
    \end{equation*}
    where
    \begin{equation*}
        w(s)=\begin{bmatrix}
            (K-1)-(\rho K-1)(1-s)\\
            (K-1)-(\rho-1)K(1-s)
        \end{bmatrix}.
    \end{equation*}
    Since $s_\dagger<s<1$, the scalar prefactor is positive, so $W(t)\geq0$ on the second branch as well.
    This proves part (b).
\end{proof}
Let $0<\rho < 1$. Then for any fixed $t$, the function $K\mapsto b_{\rho,K}(t)$ is increasing on $[\rho,1]$. Consequently, the tightest upper bound occurs at $K=\rho$. Specifically, if $K=\sbd(\|A\|^{2})$ as in Proposition \ref{prop:main_result_K}, then $K=\rho$ is equivalent to $\|A\|=0$, i.e.\ $A=0$. Conversely, as $\|A\|$ increases, the value of $K$ increases, resulting in a looser upper bound $b_{\rho,K}(t)$. 

\begin{figure}[H]
\begin{tikzpicture}[declare function={
    func(\x) = (0.5*0.5 - (0.5*0.5-1)*\x);
    refin(\x) = (0.5*0.7 - (0.5*0.7-1)*\x);
    }]
    \begin{axis}[
    axis lines = middle,
    axis on top = true,
    xlabel = {$t$},
    ylabel = {$b_{\rho,K}(t)$},
    domain=0:1,
    samples = 100,
    xmin = 0, xmax = 1.15,
    ymin = 0, ymax = 1.15,
    xtick = {1},
    xticklabels = {1},
    extra x ticks = {0},
    extra x tick labels = {$0$},
    extra x tick style = {grid = none},
    ytick = {1},
    yticklabels = {$1$},
    extra y ticks = {0.25, 0.5*0.7,0.5,0.7},
    extra y tick labels = {$\rho^2$, $\rho K$, $\rho$, $K$},
    extra y tick style = {grid = none},
    grid =  both,
    xmajorgrids = false,
    grid style = {thin, dashed, gray!50},
    thick,
    clip = false,
    legend style={ at={(1, 0.2174)}, anchor=east}
    ]

    \addplot[blue,thick]
    {func(x)};
    \addlegendentry{$K=\rho$}
    \addplot[red, thick]
      {refin(x)};
      \addlegendentry{$K>\rho$}

    \draw[thin, dashed, gray!50] (axis cs:1,0) -- (axis cs:1,1);
    
    \end{axis}
\end{tikzpicture}
\caption{The graphs of $b_{\rho,K}$ for $0<\rho < 1$ and $K=\sbd(\|A\|^{2})$ in the cases $K=\rho$ (i.e.\ $\|A\|=0$) and $K>\rho$ (i.e.\ $\|A\|>0$).}
\end{figure}
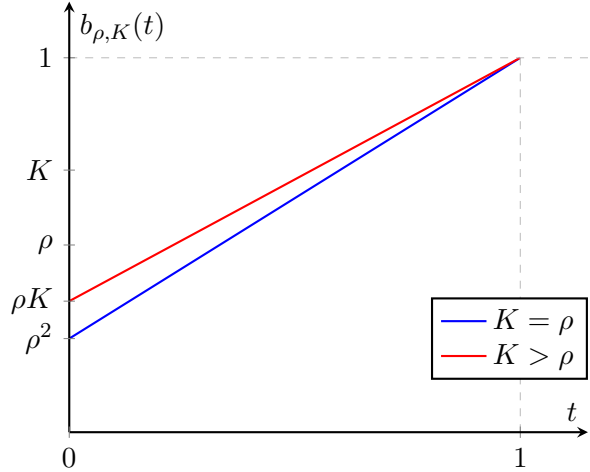

Now let $\rho>1$. Then for any fixed $t$, the function $K\mapsto b_{\rho,K}(t)$ is increasing on $[1,\rho]$. The loosest upper bound occurs at $K=\rho$. If $K=\sbd(m(A)^{2})$ as in Proposition \ref{prop:main_result_K}, then $K=\rho$ is equivalent to $m(A)=0$. Conversely, as $m(A)$ increases, the value of $K$ decreases, resulting in a tighter upper bound $b_{\rho,K}(t)$. 

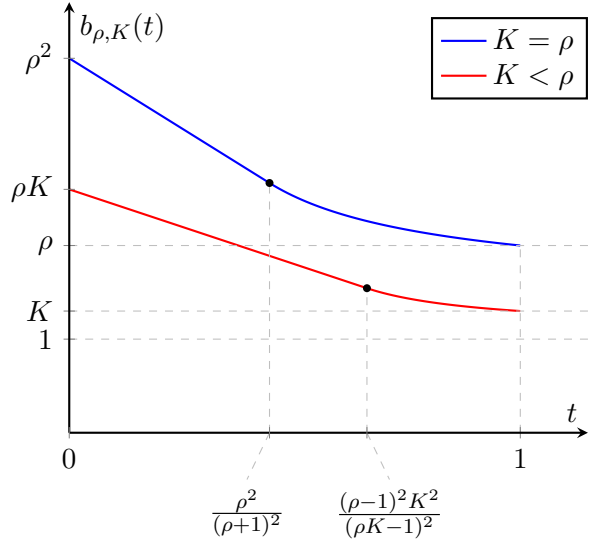
\begin{figure}[H]
\begin{tikzpicture}[declare function={
    func(\x) = (\x <= 0.444) * (4-3*\x) +
              (\x > 0.444) * (4*sqrt(\x) / (3*sqrt(\x)-1)); 
    refin(\x) = (\x <= 0.660) * (2*1.3 - (2*1.3-1)*\x) + (\x > 0.660) * ( 1 + ( (1.3-1)^2 * (1+sqrt(\x)) ) / ( 2*(1.3-1) - (2*1.3-1)*(1-sqrt(\x)) ) );
    }]
    \begin{axis}[
    axis lines = middle,
    axis on top = true,
    xlabel = {$t$},
    ylabel = {$b_{\rho,K}(t)$},
    domain=0:1,
    samples = 100,
    xmin = 0, xmax = 1.15,
    ymin = 0, ymax = 4.60,
    xtick = {1},
    xticklabels = {1},
    extra x ticks = {0, 0.444, 0.660},
    extra x tick labels = {$0$, \null, \null},
    extra x tick style = {grid = none},
    ytick = {1, 1.3, 2},
    yticklabels = {$1$, $K$, $\rho$},
    extra y ticks = {2*1.3, 4},
    extra y tick labels = {$\rho K$, $\rho^2$},
    extra y tick style = {grid = none},
    grid =  both,
    xmajorgrids = false,
    grid style = {thin, dashed, gray!50},
    thick,
    clip = false,
    legend style={ at={(1, 0.8696)}, anchor=east}
    ]
    \addplot[blue, thick]
      {func(x)};
      \addlegendentry{$K=\rho$}
    \addplot[red, thick]
      {refin(x)};
      \addlegendentry{$K<\rho$}
    
\draw[thin, dashed, gray!50] (axis cs:0.444,0) -- 
(axis cs: 0.444, {func(0.444)});
\draw[thin, dashed, gray!50] (axis cs:0.444,0) -- (axis cs:0.444-0.05,-0.5) 
    node[below, black] {$\frac{\rho^2}{(\rho+1)^2}$};

\draw[thin, dashed, gray!50] (axis cs:0.660,0) -- 
(axis cs: 0.660, {refin(0.660)});
\draw[thin, dashed, gray!50] (axis cs:0.660,0) -- (axis cs:0.660+0.05,-0.5) 
    node[below, black] {$\frac{(\rho-1)^2 K^2}{(\rho K-1)^2}$};

\draw[thin, dashed, gray!50] (axis cs:1,0) -- (axis cs:1,2);

\fill [black] (axis cs:0.444, {func(0.444)}) circle (1.5pt);
\fill [black] (axis cs:0.660, {refin(0.660)}) circle (1.5pt);

    \end{axis}
\end{tikzpicture}
\caption{The graphs of $b_{\rho,K}$ for $\rho>1$ and $K=\sbd(m(A)^{2})$ in the cases $K=\rho$ (i.e.\ $m(A)=0$) and $K<\rho$ (i.e.\ $m(A)>0$).}
\end{figure}

\begin{remark}
    We briefly comment on the choices of parameters in the proof of Proposition \ref{prop:main_result_K} in the case $\rho>1$.
    Essentially, we minimized $b(t)$ subject to the constraints given by Lemma \ref{lem:PSD_reduction_full} (i.e.\ there exists a $\delta\geq0$ and a function $x\colon\sigma(A^{*}A)\to[-1,1]$ such that $W(t)\geq0$ for all $t\in\sigma(A^{*}A)$). In particular, we require the upper-left entry of $W(t)$ to be non-negative, which establishes a lower
    bound on $b(t)$ (depending on $\delta$), namely
    \begin{align*}
        b_{*}(t)=t + \delta(1-t).
    \end{align*}
    When can we achieve this minimum?
    For $W(t)$ to remain positive semidefinite, the off-diagonal entries of $W(t)$ must be zero as well, which in turn requires the auxiliary $x(t)$ to take a specific value, namely
    \begin{align*}
       x(t)= -\frac{\rho(\delta-1)}{(\rho-1)\delta}\sqrt{t}. 
    \end{align*}
    Moreover, the lower-right entry of $W(t)$ needs to be non-negative, which means
        \begin{align*}
            \frac{\delta-1}{\delta}t+\frac{\delta}{\rho^{2}}-1\geq0.
        \end{align*}
        We assume $\delta \ge 1$, so this expression increases with respect to $t$. Thus, we may guarantee this constraint for all $t\in\sigma(A^{*}A)$ by equating it to zero at the smallest eigenvalue $t_{0}=m(A)^{2}$. Solving this for $\delta$, we lock in the value $\delta=\rho \sbd(m(A)^{2}) = \rho K \ge 1$.
        Finally, we also need to ensure that $x(t) \in [-1,1]$. 
        With our choice of $\delta \ge 1$, we have $x(t) \le 0$, so the relevant constraint is $x(t) \ge -1$, which is equivalent to $t \le \big( \frac{(\rho-1) K}{\rho K - 1} \big)^2$.
        This fully determines the first branch of $b(t)$.

    On the second branch, we must have $b(t) > b_*(t)$.
    The remaining constraints are therefore $\det(W(t)) \ge 0$ and $x(t) \in [-1,1]$.
    The smallest value of $b(t)$ then turns out to be the solution of $\det(W(t)) = 0$ with the choice $x(t) = -1$.
    This yields the formula for the second branch.

\end{remark}

\section{Proofs of main results}

We first turn to the proof of Theorem \ref{thm:main_result}, which quickly follows by combining Lemma \ref{lem:vanishing_at_zero} and Proposition \ref{prop:main_result_K}.

\begin{proof}[Proof of Theorem \ref{thm:main_result}]
    By considering $rF$ for $0<r<1$ and letting $r\nearrow1$, we may assume that $\|A\|<1$; this reduction also uses continuity of $\sbd$ and continuity of the continuous functional calculus.
    Write $F=M_A\circ G$, where $G\in M_n(\mathcal{O}(\overline{\mathbb{D}}))$ satisfies $\|G\|_\infty\leq1$ and $G(0)=0$.
    By Lemma \ref{lem:vanishing_at_zero}, the operator $G(T)$ belongs to $\superrho$. Thus Proposition \ref{prop:main_result_K}, applied to $G(T)$ and $A$, gives
    \begin{equation*}
        F(T)^*F(T)=M_A(G(T))^*M_A(G(T))\leq \mbd(A^*A),
    \end{equation*}
    as desired.
\end{proof}

We are now ready for the proof of Theorem \ref{thm:main_intro}.
\begin{proof}[Proof of Theorem \ref{thm:main_intro}]
    Let $A=F(0)$. We first consider the case $0 < \rho \le 1$.
    Let $K=\sbd(\|A\|^2)$. By Theorem \ref{thm:main_result},
    \begin{equation*}
        F(T)^*F(T)\leq \mbd(A^*A) \le \mbd(\|A\|^2) I,
    \end{equation*}
    since $\mbd$ is increasing.
    By \eqref{eqn:quadratic_equation_c_rho}, we have
    \begin{equation*}
        \mbd(\|A\|^2)=\rho K-(\rho K-1)\|A\|^2=K^2.
    \end{equation*}
    Hence $\|F(T)\|\leq K=\sbd(\|F(0)\|^2)$.

    Next, let $\rho > 1$ and let $K=\sbd(m(A)^2)$. Again by Theorem \ref{thm:main_result},
    \begin{equation*}
        F(T)^*F(T)\leq \mbd(A^*A) \le \mbd(m(A)^2) I,
    \end{equation*}
    since $\mbd$ is decreasing.
    Moreover, $m(A)^2$ lies in the affine branch of $\mbd$ by Lemma \ref{lem:threshold}.
    The same lemma also gives
    \begin{equation*}
        \mbd(m(A)^2)=\rho K-(\rho K-1)m(A)^2=K^2.
    \end{equation*}
    Hence $\|F(T)\|\leq K=\sbd(m(F(0))^2)$.
\end{proof}

\section{Sharpness}
\label{sec:sharpness}

In this section, we discuss to what extent our results are sharp. We begin with the simple observation that Theorem \ref{thm:main_intro}
is sharp, essentially because the scalar bound of \cite{SV25} is.

\begin{proposition}
  \label{prop:sharpness_main_intro}
  Let $\rho > 0$, $n \in \mathbb{N}$ and $A \in M_n(\mathbb{C})$ with $\|A \| \le 1$. Then there exist $T \in C_\rho$
  and $F \in M_n(\mathcal{O}(\overline{\mathbb{D}}))$ with $\|F\|_\infty = 1$ and $F(0) = A$ such that
  \begin{equation*}
    \|F(T)\| =
    \begin{cases}
      \sbd(\|A\|^2) & \text{if } 0 < \rho \le 1, \\
      \sbd(m(A)^2) & \text{if } \rho > 1.
    \end{cases}
  \end{equation*}
\end{proposition}

\begin{proof}
  Let
  \begin{equation*}
    T =
    \begin{bmatrix}
      0 & \rho \\ 0 & 0
    \end{bmatrix}.
  \end{equation*}
  It is known and easy to check that $T \in C_\rho$. (For instance, a unitary $\rho$-dilation of $T$ is given by
  the bilateral shift on $\ell^2(\mathbb{Z})$, where we identify $\mathbb{C}^2$ with the span of the standard basis vectors $e_0$ and $e_1$.)

  Let $A=U\operatorname{diag}(s_1,\ldots,s_n)V^*$ be a singular value decomposition of $A$,
  where $0\leq s_j\leq1$ for all $j$. As in \cite{SV25}, we consider for $0 \le s < 1$ the Möbius maps
  \begin{equation*}
    \varphi_s(z)=\frac{s+z}{1+sz},
  \end{equation*}
  and let $\varphi_1(z)=1$.
  Define
  \begin{equation*}
    F(z)=U\operatorname{diag}(\varphi_{s_1}(z),\ldots,\varphi_{s_n}(z))V^*.
  \end{equation*}
  Then $\|F\|_\infty=1$ and $F(0)=A$.

  For $0 \le s \le 1$, we have
  \begin{equation*}
    \|\varphi_s(T)\| = \Big\| 
    \begin{bmatrix}
      s & \rho(1-s^2)\\
      0 & s
    \end{bmatrix} \Big\|
    = \sbd(s^2).
  \end{equation*}
  Hence
  \begin{equation*}
    \|F(T)\|=\max_{1\leq j\leq n}\|\varphi_{s_j}(T)\|
    =\max_{1\leq j\leq n}\sbd(s_j^2).
  \end{equation*}
  Since $\sbd$ is increasing for $0<\rho\leq1$ and decreasing for $\rho>1$, this maximum is
  $\sbd(\|A\|^2)$ in the first case and $\sbd(m(A)^2)$ in the second case.
\end{proof}

The question of sharpness of Theorem \ref{thm:main_result} is more subtle, because Hermitian matrices do not form a lattice under the usual order. For instance, there is no smallest diagonal matrix dominating
the matrix $\begin{bmatrix} 1 & 1 \\ 1 & 1 \end{bmatrix}$.
It is therefore not clear if there is an optimal right-hand side in Theorem \ref{thm:main_result}.
Nonetheless, we will show that the right-hand side in Theorem \ref{thm:main_result} is optimal subject to some natural assumptions.

We begin with the following lower bound, which shows in particular that the estimate $F(T)^* F(T) \le \sbd(F(0)^* F(0))^2$, which one might naively conjecture in view of Theorem \ref{thm:main_intro}, is not true in general.

\begin{proposition}
  \label{prop:lower_bound}
  Let $\rho > 0$ and let $h\colon [0,1] \to [0,\infty)$ be a continuous function such that
  $F(T)^* F(T) \le h(F(0)^* F(0))$ for all $F \in M_n(\mathcal O(\overline{\mathbb D}))$ with $\|F\|_\infty \le 1$, all $n \in \mathbb N$ and all $T \in C_\rho$. Then
  \begin{equation*}
    h(t) \ge \rho^2 - (\rho^2 - 1) t \qquad \text{ for all } t \in [0,1].
  \end{equation*}
\end{proposition}

\begin{proof}
  Let $0\leq s<1$ and set
  \begin{equation*}
    T =
    \begin{bmatrix}
      0 & \rho \\ 0 & 0
    \end{bmatrix},
    \qquad
    G(z) =
    \begin{bmatrix}
      0 & z \\ z & 0
    \end{bmatrix},
    \qquad
    A =
    \begin{bmatrix}
      s & 0 \\ 0 & 0
    \end{bmatrix}.
  \end{equation*}
  As remarked in the proof of Proposition \ref{prop:sharpness_main_intro}, we have $T \in C_\rho$.
  Let
  \begin{equation*}
    F=M_A\circ G,
  \end{equation*}
  hence $\|F\|_\infty\leq1$ and $F(0)=A$.

  We identify $F(T)$ as an operator on $\mathbb{C}^2\oplus\mathbb{C}^2$. A direct computation gives
  \begin{equation*}
    F(T)=
    \begin{bmatrix}
      sI & \sqrt{1-s^2}\,T\\
      \sqrt{1-s^2}\,T & 0
    \end{bmatrix}.
  \end{equation*}
  Let $e_1,e_2$ be the standard basis of $\mathbb{C}^2$ and put $\xi=(e_2,0)\in\mathbb{C}^2\oplus\mathbb{C}^2$.
  Then
  \begin{equation*}
    \langle F(T)^*F(T)\xi,\xi\rangle
    =\|F(T)\xi\|^2
    =s^2+(1-s^2)\|Te_2\|^2
    =\rho^2-(\rho^2-1)s^2.
  \end{equation*}
  On the other hand,
  \begin{equation*}
    \langle h(A^*A)\xi,\xi\rangle=h(s^2).
  \end{equation*}
  Now, if $F(T)^*F(T)\leq h(A^*A)$, we must have
  \begin{equation*}
    h(s^2)\geq \rho^2-(\rho^2-1)s^2
  \end{equation*}
  for every $0\leq s<1$, which gives the desired lower bound on $h$.
\end{proof}

If $\rho > 1$, then the lower bound of Proposition \ref{prop:lower_bound} is larger than the affine branch of the upper bound in Theorem \ref{thm:main_result}. This is not a contradiction since the upper bound in Theorem \ref{thm:main_result} also depends on $m(A)$. Indeed, in the worst case when $m(A) = 0$, the lower and the upper bound agree.

We now turn to showing sharpness of Theorem \ref{thm:main_result}
under suitable assumptions. As in Proposition \ref{prop:lower_bound}, we only consider bounds of the form $h(F(0)^* F(0))$ for some continuous function $h$.
Given $\rho < 1$ and $K \in [\rho,1)$, let $t_0 = \frac{K (\rho - K)}{\rho K - 1}$,
so $K = \sbd(t_0)$ by Lemma \ref{lem:threshold} (as $K$ lies in the image of $\sbd$) and consider
\[
\begin{aligned}
  \mathcal F_{\rho,K}
  = \{ h \in C_{\mathbb R}([0,1]):{}&
      F(T)^* F(T) \le h(F(0)^* F(0)) \text{ for all } T \in C_\rho \text{ and all } \\
    & F \in M_n(\mathcal{O}(\overline{\mathbb{D}})) 
      \text{ with } \|F\|_\infty \le 1 \text{ and } \|F(0)\|^2 = t_0 \text{ and all } n \in \mathbb N \}.
\end{aligned}
\]
Similarly, if $\rho > 1$ and $K \in (1,\rho]$,  define $t_0$ as before and consider
\[
\begin{aligned}
  \mathcal F_{\rho,K}
  = \{ h \in C_{\mathbb R}([0,1]):{}&
      F(T)^* F(T) \le h(F(0)^* F(0)) \text{ for all } T \in C_\rho \text{ and all } \\
    &F \in M_n(\mathcal{O}(\overline{\mathbb{D}}))
      \text{ with } \|F\|_\infty \le 1 \text{ and } m(F(0))^2 = t_0 \text{ and all } n \in \mathbb N \}.
\end{aligned}
\]
Theorem \ref{thm:main_result} shows that $b_{\rho,K} \in \mathcal F_{\rho,K}$,
and we think of $\mathcal F_{\rho,K}$ as the class of all functions that can serve as upper bounds in Theorem \ref{thm:main_result} for the parameters $\rho$ and $K$.

Now Proposition \ref{prop:sharpness_main_intro} implies that every $h \in \mathcal F_{\rho,K}$ satisfies 
\begin{equation*}
  h(t_0) \ge \sbd(t_0)^2 = K^2.
\end{equation*}
Moreover, for an upper bound $h \in \mathcal F_{\rho,K}$ to imply the estimate of Theorem \ref{thm:main_intro}, we must have equality. Hence, we consider $h \in \mathcal{F}_{\rho,K}$ such that $h(t_0) = K^2$. Finally, note that for functions $h \in \mathcal{F}_{\rho,K}$, only
the values of $h$ on $[t_0,1]$ (respectively on $[0,t_0])$ are relevant if $\rho > 1$ (respectively if $\rho < 1$).
With these observations in mind, we have the following result, which shows that Theorem \ref{thm:main_result} is optimal given our assumptions.

\begin{theorem}
  \label{thm:sharpness_main_result}
  Let $\rho >0$, $\rho \neq 1$ and $K \in (1,\rho]$ if $\rho > 1$ or $K \in [\rho,1)$ if $\rho < 1$.
  Let $t_0 = \frac{K (\rho - K)}{\rho K - 1}$ and let $\mathcal F_{\rho,K}$ be as above.
  If $h \in \mathcal F_{\rho,K}$ satisfies $h(t_0) = K^2$, then
  \begin{equation*}
    h(t) \ge b_{\rho,K}(t) \qquad \text{ for all } t \in
    \begin{cases}
      [0,t_0] & \text{if } 0 < \rho < 1, \\
      [t_0,1] & \text{if } \rho > 1.
    \end{cases}
  \end{equation*}
  Here, $b_{\rho,K}$ is the function of Theorem \ref{thm:main_result}.
\end{theorem}

\begin{proof}
  We first consider the affine branch of $b_{\rho,K}$.
  Let $t \in [0,t_0]$ if $0 < \rho < 1$ or $t \in [t_0,1]$ if $\rho > 1$.
  By continuity, we may assume that $t < 1$.
  Let
  \begin{equation*}
    T =
    \begin{bmatrix}
      0 & \rho \\ 0 & 0
    \end{bmatrix},
    \qquad
    G(z) =
    \begin{bmatrix}
      0 & z \\ z & 0
    \end{bmatrix},
    \qquad
    A =
    \begin{bmatrix}
      \sqrt{t} & 0 \\ 0 & \sqrt{t_0}
    \end{bmatrix}.
    \qquad
  \end{equation*}
  As before, $T\in C_\rho$. Let $F=M_A\circ G.$ Then $\|F\|_\infty\leq1$ and $F(0)=A$. Moreover, in the case $0<\rho<1$, we have
  $\|A\|^2=t_0$, and in the case $\rho>1$, we have $m(A)^2=t_0$. Thus,
  \begin{equation}
    \label{eq:sharpness_main_result}
    F(T)^* F(T) \le h(A^*A).
  \end{equation}

 A direct computation yields
  \begin{equation*}
    F(T)=
    \begin{bmatrix}
      \sqrt{t}\,I & \sqrt{(1-t)(1-t_0)}\,T\\
      \sqrt{(1-t)(1-t_0)}\,T & \sqrt{t_0}\,I
    \end{bmatrix}.
  \end{equation*}
  Let $e_1,e_2$ denote the standard basis of $\mathbb C^2$.
  The compression of $F(T)^*F(T)$ onto the span of $(e_2,0)$ and $(0,e_1)$ in $\mathbb{C}^2 \oplus \mathbb{C}^2$ is given by
  \begin{equation*}
    \begin{bmatrix}
      t+\rho^2(1-t)(1-t_0) & \rho\sqrt{t_0(1-t)(1-t_0)}\\
      \rho\sqrt{t_0(1-t)(1-t_0)} & t_0
    \end{bmatrix}.
  \end{equation*}
  On the other hand, the corresponding compression of $h(A^*A)$ is
  \begin{equation*}
    \begin{bmatrix}
      h(t) & 0\\
      0 & h(t_0)
    \end{bmatrix}
    =
    \begin{bmatrix}
      h(t) & 0\\
      0 & K^2
    \end{bmatrix}.
  \end{equation*}
  Hence \eqref{eq:sharpness_main_result} implies that
  \begin{equation*}
    \begin{bmatrix}
      h(t)-t-\rho^2(1-t)(1-t_0) & -\rho\sqrt{t_0(1-t)(1-t_0)}\\
      -\rho\sqrt{t_0(1-t)(1-t_0)} & K^2-t_0
    \end{bmatrix}
    \geq0.
  \end{equation*}
  Since $K=\sbd(t_0)$, we have $K^2 - t_0 = \rho K (1 - t_0)$ by \eqref{eqn:quadratic_equation_c_rho}. Substituting this into the matrix and taking its determinant implies
  \begin{equation*}
    \big(h(t)-t-\rho^2(1-t)(1-t_0)\big)\rho K(1-t_0)
    \geq \rho^2t_0(1-t)(1-t_0).
  \end{equation*}
  Thus
  \begin{equation*}
    h(t)
    \geq t+\rho^2(1-t)(1-t_0)+\frac{\rho t_0}{K}(1-t)
     = \rho K-(\rho K-1)t,
  \end{equation*}
  where the equality again follows from the quadratic equation for $K$.
  This establishes the desired lower bound on $h$ in the affine branch
  and completes the proof in the case $\rho < 1$.

  Let $\rho > 1$. We now consider the non-affine branch of $b_{\rho,K}$. 
  Let $s_\dagger = \frac{(\rho - 1) K}{\rho K - 1}$, and assume that $\sqrt{t} > s_\dagger$.
  We use the same $A$, $G$ and $F = M_A \circ G$ as before, but for $T$, we use the $3 \times 3$ truncated weighted shift
  \begin{equation*}
    T =
    \begin{bmatrix}
      0 & \sqrt{x} & 0\\
      0 & 0 & \sqrt{y}\\
      0 & 0 & 0
    \end{bmatrix},
  \end{equation*}
  where
  \begin{equation*}
    x = \rho \frac{1 + \rho \lambda}{1 + \lambda} \qquad \text{ and } \qquad y = \frac{\rho}{1 + (\rho - 1) \lambda}.
  \end{equation*}
  Here, $\lambda \in [0,\infty)$ is a parameter that will be chosen later. Observe that $\rho \le x < \rho^2$
  and $0 < y \le \rho$ and $\rho^3 - \rho(x+y) + (2 - \rho) xy = 0$. Hence $T \in C_\rho$ by Lemma \ref{lem:weighted_shift_C_rho} below.
  (This family interpolates between the case $x= y = {\rho}$ (corresponding to $\lambda = 0$) and the case $x = \rho^2$, $y = 0$
  (corresponding to $\lambda \to \infty$).)
  
  A computation shows that
  \begin{equation*}
    F(T)=
    \begin{bmatrix}
      \sqrt{t}\,I-\sqrt{t_0}(1-t)T^2&
      \sqrt{(1-t)(1-t_0)}\,T\\
      \sqrt{(1-t)(1-t_0)}\,T&
      \sqrt{t_0}\,I-\sqrt{t}(1-t_0)T^2
    \end{bmatrix}.
  \end{equation*}
  The compression of $F(T)^* F(T)$ onto the span of $(e_2,0)$, $(0,e_1)$ and $(0,e_3)$ in $\mathbb{C}^3 \oplus \mathbb{C}^3$ is given by
  \begin{equation*}
    \begin{bmatrix}
      t+(1-t)(1-t_0)x&
      \sqrt{t_0(1-t)(1-t_0)x}&
      \sqrt{t(1-t)(1-t_0)y}\big(1-(1-t_0)x\big)\\
      \sqrt{t_0(1-t)(1-t_0)x}&
      t_0&
      -\sqrt{tt_0}(1-t_0)\sqrt{xy}\\
      \sqrt{t(1-t)(1-t_0)y}\big(1-(1-t_0)x\big)&
      -\sqrt{tt_0}(1-t_0)\sqrt{xy}&
      t_0+(1-t)(1-t_0)y+t(1-t_0)^2xy
    \end{bmatrix}.
  \end{equation*}
  The corresponding compression of $h(A^* A)$ is
  $\diag(h(t),K^2,K^2)$. Thus $F(T)^* F(T) \le h(A^* A)$ implies that
  \begin{equation}
    \label{eqn:sharpness_non_affine}
    \begin{bmatrix}
      h(t)-t-(1-t)(1-t_0)x & -v^*\\
      -v & C
    \end{bmatrix}
    \ge 0,
  \end{equation}
  where
  \begin{equation*}
    v =
    \begin{bmatrix}
      \sqrt{t_0(1-t)(1-t_0)x}\\
      \sqrt{t(1-t)(1-t_0)y}\big(1-(1-t_0)x\big)
    \end{bmatrix}
  \end{equation*}
  and
  \begin{equation*}
    C =
    \begin{bmatrix}
      K^2-t_0 & \sqrt{tt_0}(1-t_0)\sqrt{xy}\\
      \sqrt{tt_0}(1-t_0)\sqrt{xy} &
      K^2-t_0-(1-t)(1-t_0)y-t(1-t_0)^2xy
    \end{bmatrix}.
  \end{equation*}
  We claim that $C$ is invertible. Indeed, using
  $K^2-t_0=\rho K(1-t_0)$ from \eqref{eqn:quadratic_equation_c_rho}, we have
  \begin{equation*}
    C=(1-t_0)
    \begin{bmatrix}
      \rho K & \sqrt{tt_0xy}\\
      \sqrt{tt_0xy} & \rho K-(1-t)y-t(1-t_0)xy
    \end{bmatrix}.
  \end{equation*}
  The determinant of this matrix is
  \begin{align*}
    \det (C)
    &=(1-t_0)^2\big(\rho K\big(\rho K-(1-t)y-t(1-t_0)xy\big)-tt_0xy\big)\\
    &=(1-t_0)^2\big(K^2(\rho^2-txy)-\rho K(1-t)y\big)\\
    &=\frac{(1-t_0)^2\rho^2 K}{(1+\lambda)(1+(\rho-1)\lambda)}
      \big((1-t)\big((K-1)+(\rho K-1)\lambda\big)
        +K(\rho-1)\lambda^2\big)>0.
  \end{align*}
  Hence $C$ is invertible. Taking the Schur complement in \eqref{eqn:sharpness_non_affine} gives
  \begin{equation*}
    h(t) \ge t+(1-t)(1-t_0)x+v^*C^{-1}v.
  \end{equation*}
  Using the formula for $t_0$ from Lemma \ref{lem:threshold} and simplifying
  yields the equivalent lower bound
  \begin{equation}
    \label{eqn:lower_bound}
    h(t) \ge
    1+\frac{(K-1)^2(1+\sqrt{t})}{d}
    -\frac{K(\rho-1)(1+\sqrt{t})(a\lambda-(K-1)(1-\sqrt{t}))^2}{\tau d} ,
  \end{equation}
  where
  \begin{align*}
    a&= (\rho K - 1) \sqrt{t} - K(\rho - 1), \\
    d&=2(K-1)-(\rho K-1)(1-\sqrt{t}),\\
    \tau&=(1-t)\big((K-1)+(\rho K-1)\lambda\big)+K(\rho-1)\lambda^2.
  \end{align*}
  Observe that the assumption $\sqrt{t} > s_\dagger$ yields $a > 0$ and $d > 0$.
  To maximize the right-hand side of \eqref{eqn:lower_bound},
  we choose $\lambda \in [0,\infty)$ so that the last summand of \eqref{eqn:lower_bound} vanishes.
  Then \eqref{eqn:lower_bound} gives $h(t) \ge \mbd(t)$, completing the proof.
\end{proof}

The following lemma was used in the last proof. The special case when $x=y={\rho}$ in the lemma
was previously established by Carrot \cite{Carrot03}.
Closely related results about membership of weighted shifts in $C_\rho$ can
also be found in the work of Eckstein and R\'{a}cz \cite{ER73}.

\begin{lemma}
  \label{lem:weighted_shift_C_rho}
  Let $\rho > 0$ and let $x,y \in [0,\rho^2]$ satisfy
  \begin{equation}
    \label{eq:cubic_equation}
    \rho^3 - \rho(x+y) + (2 - \rho) x y = 0.
  \end{equation}
  Then
  \begin{equation*}
    \begin{bmatrix}
        0&\sqrt{x}&0\\
        0&0&\sqrt{y}\\
        0&0&0
    \end{bmatrix} \in C_\rho.
  \end{equation*}
\end{lemma}

\begin{proof}
  Let $T$ denote the $3 \times 3$ matrix in the statement. We will check condition \eqref{eq:rho_contraction_characterization} for $T$.
  Since $T$ is a truncated weighted shift, $\omega T$ is unitarily equivalent to $T$ for every $\omega \in \mathbb{C}$ with $|\omega| = 1$.
  Thus, it suffices to check \eqref{eq:rho_contraction_characterization} for $z = r \in [0,1]$. We have
  \[
    I - 2 \Big(1 - \frac{1}{\rho}\Big) r \Re(T) -\Big( \frac{2}{\rho} - 1\Big) r^2 T^* T
     =
    \begin{bmatrix}
     1&-\Big(1-\dfrac{1}{\rho}\Big)r\sqrt{x}&0\\[2mm]
     -\Big(1-\dfrac{1}{\rho}\Big)r\sqrt{x}&
     1-\Big(\dfrac{2}{\rho}-1\Big)r^2x&
     -\Big(1-\dfrac{1}{\rho}\Big)r\sqrt{y}\\[2mm]
     0&-\Big(1-\dfrac{1}{\rho}\Big)r\sqrt{y}&
     1-\Big(\dfrac{2}{\rho}-1\Big)r^2y
     \end{bmatrix},  \]
     and we have to show that this matrix is positive semidefinite for all $r \in [0,1]$.
     We will use Sylvester's criterion to show that the matrix is positive definite for all $r \in [0,1)$, which will imply the result.

     The first two leading principal minors are $1$ and $\frac{\rho^2 - r^2 x}{\rho^2}$,
     which are strictly positive for all $r \in [0,1)$ as $x \le \rho^2$.

     It remains to check the determinant. Using \eqref{eq:cubic_equation} to eliminate the factor of $x+y$, we can rewrite the determinant as
     \begin{equation*}
       \frac{1}{\rho^3} \big(\rho^3 - \rho r^2 (x+y) + (2-\rho) r^4 x y\big)
       = \frac{1-r^2}{\rho^3} \big( \rho^3 - (2-\rho) r^2 x y \big).
     \end{equation*}
     If $\rho \ge 2$, then this expression is clearly positive. If $ 0 < \rho < 2$, then since
     $x y \le \rho^4$,
     \begin{equation*}
       \rho^3 - (2-\rho) r^2 x y > \rho^3 - (2-\rho) \rho^4
       = \rho^3 (\rho - 1)^2 \ge 0. \qedhere
     \end{equation*}
\end{proof}

\bibliographystyle{amsplain}
\bibliography{literature}

\end{document}